\documentclass[11pt,reqno]{amsart}
\usepackage{amssymb, amsmath, amsthm, mathrsfs, setspace, enumerate}
\usepackage[bookmarksnumbered, colorlinks, plainpages]{hyperref}
\usepackage{array}
\usepackage{tabularx}
\usepackage{tikz}
\usepackage{mathrsfs}
\newtheorem*{theoA}{Theorem A}
\newtheorem*{theoB}{Theorem B}

\newtheorem{theo}{Theorem}[section]
\newtheorem{lem}{Lemma}[section]
\newtheorem{cor}{Corollary}[section]

\newcommand{\ol}{\overline}

\newcommand{\be}{\begin{equation}}
\newcommand{\ee}{\end{equation}}
\newcommand{\beas}{\begin{eqnarray*}}
\newcommand{\eeas}{\end{eqnarray*}}
\newcommand{\bea}{\begin{eqnarray}}
\newcommand{\eea}{\end{eqnarray}}

\numberwithin{equation}{section}
\begin{document}

\title[M\MakeLowercase{eromorphic Functions and Linearization Phenomena in Partial Differential}......]{\LARGE 
M\Large \MakeLowercase{eromorphic Functions and Linearization Phenomena in Partial Differential Equations}}

\date{}
\author[S. M\MakeLowercase{ajumder}, D. P\MakeLowercase{ramanik and} J. B\MakeLowercase{anerjee}]{S\MakeLowercase{ujoy} M\MakeLowercase{ajumder}$^1$, D\MakeLowercase{ebabrata} P\MakeLowercase{ramanik }$^2$ \MakeLowercase{and}  J\MakeLowercase{hilik} B\MakeLowercase{anerjee}$^3$ $^{*}$}
\address{$^{1}$ Department of Mathematics, Raiganj University, Raiganj, West Bengal-733134, India.}
\email{sm05math@gmail.com, sjm@raiganjuniversity.ac.in}
\address{$^{2}$ Department of Mathematics, Raiganj University, Raiganj, West Bengal-733134, India.}
\email{debumath07@gmail.com}
\address{$^{3}$ Department of Mathematics, University of Kalyani, West Bengal 741235, India.}
\email{jhilikbanerjee38@gmail.com}

\renewcommand{\thefootnote}{}
\footnote{*\emph{Corresponding Author}: Jhilik Banerjee.}

\renewcommand{\thefootnote}{\arabic{footnote}}
\setcounter{footnote}{0}

\begin{abstract} 
In this paper, we investigate meromorphic solutions of certain nonlinear partial differential equations in several complex variables involving differential and functional operators. Let $f$ be a non-constant meromorphic function in $\mathbb{C}$, $g$ an entire function in $\mathbb{C}^n$, and $h(z)=f(z_1+z_2+\ldots+z_n)$. We study the equations
\begin{align*}
\frac{\partial h(z)}{\partial z_i}=a G^g_{h}(z)+bh(z)+c\;\;\text{and}\;\;\frac{\partial h(z)}{\partial z_i}=a(z)G^g_{h}(z)+b(z)h(z)+c(z),
\end{align*}
where $z\in\mathbb{C}^n$, $i\in\{1,2,\ldots,n\}$, $a(\neq 0), b, c\in\mathbb{C}$ or $a(z)(\not\equiv 0), b(z),c(z)$ are polynomials in $\mathbb{C}^n$, and $G^g_h(z)=h(g(z),g(z),\ldots,g(z))$.
The results obtained in the paper, extend previous studies on meromorphic solutions of functional-differential equations to the setting of several complex variables, and further illustrate the rigidity imposed by value distribution properties on nonlinear functional equations.
\end{abstract}

\thanks{Typeset by \AmS -\LaTeX}
\maketitle

\section{{\bf Introduction and main results}}

Functional-differential equations involving compositions, derivatives, and analytic functions arise naturally in many areas of Complex Analysis and differential equations. Equations of the form
\begin{align}\label{I1.1}
f'(x)=af(g(x))
\end{align}
and its generalization
\begin{align}\label{I1.2}
f'(x)=af(g(x))+bf(x)+c
\end{align}
where $a\neq 0$, while $b$ and $c$ are constants or, more generally, functions combine both differential and functional behaviors, making their analysis significantly more complicated than ordinary differential equations. Such equations have attracted considerable attention because they connect the growth, value distribution, and structural properties of analytic functions with the dynamics generated by the function $g$. Functional-differential equations given by (\ref{I1.2}) arise naturally in several branches of mathematics and applied sciences. Such equations appear in the theory of boundary value problems for hyperbolic partial differential equations and include many important special cases. One notable example is the pantograph equation
\begin{align*}
f'(x)=af(\alpha x)+bf(x),
\end{align*}
where $a$, $b$, and $\alpha$ are constants. Pantograph-type equations have numerous applications, including models of cell growth, current collection systems for electric locomotives, and wavelet theory (see, for example, \cite{O-T-1971}, \cite{W-C-K-B-2000}, and the references therein), and they have been extensively studied for both real and complex variables by many authors.

Another important example occurs when
\begin{align*}
g(x)=x-k,
\end{align*}
with $k>0$ fixed, equation (\ref{I1.1}) reduces to the well-known and extensively studied linear differential-difference (or delay-differential) equation (see \cite{B-C-1963}). In \cite{U-1965}, Utz posed the problem of existence for equation (\ref{I1.1}). Siu \cite{S-1965} established global existence and uniqueness results for (\ref{I1.1}) under certain conditions on the function g and the constant a. Local existence and uniqueness problems for more general equations were later investigated by Anderson \cite{A-1966}, and by Oberg in \cite{O-1969} for local real solutions and in \cite{O-1971} for local complex solutions. Related equations in the complex domain were also studied in \cite{B-1997}, \cite{BMW-2004}, \cite{D-1995}, \cite{D-I-1997}, \cite{G-1967}, \cite{G-Y-1973}, and related works.

It was shown in \cite{G-1967} that if $f$ is a non-constant entire function and $g$ is an entire function on the complex plane $\mathbb{C}$ satisfying equation (\ref{I1.1}), with a constant, then $g$ must be linear. This conclusion was recently extended in \cite{BMW-2004} to equations of the form
\begin{align*}
f'=af(g)+bf,
\end{align*}
where $a\neq 0$ and $b$ is constant; in this case the equations reduce to those of pantograph type. However, no analogous characterization is currently known when $f$ is meromorphic in $\mathbb{C}$. It is known only from \cite{G-Y-1973} that $g$ must be a polynomial whenever $f$ is a transcendental meromorphic function in $\mathbb{C}$. 
Despite the extensive literature on these equations, and despite occasional overlap among existing results, several important questions remain unresolved. In 2007, Li and Saleeby \cite{Li-Saleeby-2007} proved that $g$ must be linear in equation (\ref{I1.2}) when $f$ is a transcendental meromorphic function in $\mathbb{C}$ and obtained the following results.

\begin{theoA}\cite[Theorem 2.1]{Li-Saleeby-2007} Suppose that $f$ is a nonconstant meromorphic function in $\mathbb{C}$ and $g$ is an entire function in $\mathbb{C}$ satisfying the equation
\begin{align*}
f'(z)=af(g(z))+bf(z)+c,
\end{align*}
where $a \neq 0$ and $b, c$ are constants. Then
\begin{enumerate}
\item[(i)] $g$ must be linear if $f$ is transcendental;
\item[(ii)] $g$ must be a polynomial of degree less than or equal to $2$ if $f$ is rational. Furthermore, $\deg g = 2$ if and only if
\begin{align*}
f(z)=\frac{\alpha}{z-w_0}+\beta,\quad \text{and}\quad
g(z)=w_0-a(z-w_0)^2,
\end{align*}
with $b=a\beta+c=0$, where $\alpha \neq 0$, $\beta$, and $w_0$ are complex numbers.
\end{enumerate}
\end{theoA}

\begin{theoB}\cite[Theorem 2.4]{Li-Saleeby-2007} Suppose that $f$ is a transcendental meromorphic function in $\mathbb{C}$ and $g$ is an entire function in $\mathbb{C}$ satisfying the equation
\begin{align*}
f'(z)=a(z)f(g(z))+b(z)f(z)+c(z),
\end{align*}
where $a\not\equiv 0$, $b$, $c$ are polynomials in $\mathbb{C}$. Then $g$ must be linear.
\end{theoB}

We define $\mathbb{Z}_+=\mathbb{Z}[0,+\infty)=\{n\in \mathbb{Z}: 0\leq n<+\infty\}$ and $\mathbb{Z}^+=\mathbb{Z}(0,+\infty)=\{n\in \mathbb{Z}: 0<n<+\infty\}$. Let $f$ be a meromorphic function on $\mathbb{C}^n$. 
We define
\begin{align*}
\displaystyle \partial_{z_i}(f(z))=\frac{\partial f(z)}{\partial z_i},\ldots,\partial^{l_i}_{z_i}(f(z))=\frac{\partial^{l_i} f(z)}{\partial z_i^{l_i}},
\end{align*}
where $l_i\in \mathbb{Z}^n_+$ and $i=1,2,\ldots,n$ and  
\begin{align*}
\displaystyle \partial^{I}(f(z))=\frac{\partial^{|I|}f(z)}{\partial z_1^{i_1}\cdots \partial z_m^{i_n}},
\end{align*}
where $I=(i_1,\ldots,i_m)\in\mathbb{Z}^n_+$ is a multi-index such that $|I|=\sum_{j=1}^n i_j$.

\medskip
Let $f:\mathbb{C}^n\to\mathbb{P}^1$ and $g:\mathbb{C}^n\to \mathbb{P}^1$ be two non-constant meromorphic functions, where $\mathbb{P}^1=\mathbb{C}\cup\{\infty\}$. Throughout the paper, we consider the meromorphic function $G^g_f:\mathbb{C}^n\to \mathbb{P}^1$ given by
\begin{align}\label{eq}
G^g_f(z)=f(g(z),g(z),\ldots,g(z)).
\end{align}

In the setting of one complex variable, many classical results describe how the nature of a meromorphic solution $f$ restricts the possible form of the entire function $g$. However, extending these problems to the framework of several complex variables introduces new analytical difficulties and richer geometric structures. Functions of several complex variables possess properties that differ fundamentally from those in one variable, particularly with respect to holomorphic extension, partial differentiation, and multidimensional complex geometry. Consequently extending such problems from one complex variable to several complex variables introduces additional challenges because holomorphic functions in $\mathbb{C}^n$ possess richer geometric and analytic structures than those in the one-dimensional setting.

\medskip
Let
\begin{align*}
h(z)=f(z_1+z_2+\ldots+z_n),\quad z=(z_1,z_2,\ldots,z_n)\in\mathbb{C}^n,
\end{align*}
where $f$ is a meromorphic function in the complex plane and g is an entire function in $\mathbb{C}^n$. Since $h$ depends only on the sum $z_1+z_2+\ldots+z_n$, the equation
\begin{align}\label{Eq}
\frac{\partial h(z)}{\partial z_i}=a G^g_{h}(z)+bh(z)+c,
\end{align}
connects the growth and differential properties of $f$ with the geometric structure of $g$, where $i\in\{1,2,\ldots,n\}$, $a\neq 0, b,c$ are constants in $\mathbb{C}$. 

The investigation of such equations is motivated by classical results on meromorphic solutions of differential equations. In Nevanlinna theory, one seeks relations between the growth of meromorphic functions and the algebraic form of the equations they satisfy. Many classical theorems show that transcendental meromorphic functions cannot satisfy overly restrictive algebraic differential relations unless the coefficients involved are very simple. Consequently, rigidity phenomena often occur: the auxiliary functions must reduce to polynomials of low degree or even linear functions.
Now we state our results.

\begin{theo}\label{t1.1} Suppose that $f$ is a non-constant meromorphic function in $\mathbb{C}$ and $g$ is an entire function in $\mathbb{C}^n$ satisfying the equation (\ref{Eq}),
where $z\in\mathbb{C}^n$, $i\in\{1,2,\ldots,n\}$, $a\neq 0, b,c$ are constants in $\mathbb{C}$, $h(z)=f(z_1+z_2+\ldots+z_n)$ and $G^g_h(z)$ is defined by (\ref{eq}). Then 
\begin{enumerate}
\item[(i)] $g$ must be linear in $\mathbb{C}^n$, if $f$ is transcendental;
\item[(ii)] $g$ must be a polynomial in $\mathbb{C}^n$ of degree less than or equal to $2$, if $f$ is a rational function in $\mathbb{C}$. Furthermore the degree of $g$ is $2$ if and only if 
\begin{align*}
f(z)=\beta+\frac{\alpha}{z-w_0},\quad ng(z)=w_0-a\left(z_1+\ldots+z_n-w_0\right)^2
\end{align*}
and $b=a\beta+c=0$, where $\alpha\neq 0, \beta, w_0$ are complex numbers.
\end{enumerate}
\end{theo}

\begin{theo}\label{t1.2} Suppose that $f$ is a transcendental meromorphic function in $\mathbb{C}$ and $g$ is an entire function in $\mathbb{C}^n$ satisfying the following equation
\begin{align}\label{Eq1}
\frac{\partial h(z)}{\partial z_i}=a(z) G^g_{h}(z)+b(z)h(z)+c(z),
\end{align}
where $z\in\mathbb{C}^n$, $i\in\{1,2,\ldots,n\}$, $a\not\equiv  0, b,c$ are polynomials in $\mathbb{C}^n$, $h(z)=f(z_1+z_2+\ldots+z_n)$ and $G^g_h(z)$ is defined by (\ref{eq}). Then $g$ must be a linear.
\end{theo}

From Theorem \ref{t1.1}, we get the following corollaries.
\begin{cor}Suppose that $f$ is a non-constant meromorphic function in $\mathbb{C}$ and $g$ is an entire function in $\mathbb{C}^n$ satisfying the equation (\ref{Eq}). If $g$ is nonlinear, then the solution $f$ must be of the form
\begin{align*}
f(z)=\beta+\frac{\alpha}{z-w_0}
\end{align*}
for some constants $\alpha,\beta, w_0$.
\end{cor}

\begin{cor}Suppose that $f$ is a meromorphic function in $\mathbb{C}$ and $g$ is an entire function in $\mathbb{C}^n$ satisfying the equation (\ref{Eq}). If $g$ is nonlinear and $b\neq 0$ then the solution $f$ must be a constant.
\end{cor}

The present study extends these investigations to the framework of several complex variables and establishes stronger rigidity results. In particular, we show that if $f$ is a transcendental meromorphic function, then the associated entire function $g$ must be linear in $\mathbb{C}^n$. If $f$ is rational, then $g$ is necessarily a polynomial of degree at most two, and the exact forms of the extremal solutions can be completely determined. These theorems extend classical functional-differential equations from one complex variable to the multidimensional setting and contribute to the understanding of the relationship between meromorphic functions, entire mappings, and partial differential equations in several complex variables.

\subsection{{\bf Basic notations in several complex variables}}
In recent years, Nevanlinna value distribution theory in several complex variables has emerged as an active and rapidly developing area of research in complex analysis. Researchers have been especially interested in extending classical results from the theory of one complex variable to higher-dimensional complex spaces. Consequently, the study of value distribution theory in several complex variables has become a central topic in contemporary complex analysis.
Recent investigations have revealed deep connections between Nevanlinna theory and various branches of analysis and geometry. In particular, significant progress has been made in its applications to complex geometry, normal family theory, linear partial differential equations, partial difference equations, partial differential-difference equations, and Fermat-type functional equations. 
The works in references \cite{Banerjee-Majumder-2026}, \cite{Cao-Korhonen-2016}, \cite{Dovbush-2021}, \cite{Hao-Zhang-2025}, \cite{LY11}-\cite{FL}, \cite{Majumder-2026}, \cite{Majumder-Das}, \cite{Majumder-Das-Pramanik-2025}, \cite{Majumder-Sarkar}, \cite{Majumder-Sarkar-2026}, \cite{Majumder-Sarkar-2027} and \cite{Majumder-Sarkar-Pramanik} provide an important foundation for understanding the current progress and research trends in Nevanlinna value distribution theory in several complex variables.

\smallskip
For any subset $A\subset \mathbb{C}^n$ and $r\ge0$, define (see \cite[pp. 6]{Stoll-1974}) $A[r]=\{z\in A:\|z\|\le r\}$,
$A(r)=\{z\in A:\|z\|<r\}$ and $A\langle r\rangle=\{z\in A:\|z\|=r\}$.
We introduce the function $\tau(z)=\|z\|^2$.
On $\mathbb{C}^n$, the exterior derivative decomposes as $d=\partial+\bar{\partial}$, and we set
$d^c=\frac{\iota}{4\pi}(\bar{\partial}-\partial)$ and $dd^c=\frac{\iota}{2\pi}\partial\bar{\partial}$.
The standard \text{K\"ahler} form on $\mathbb{C}^n$ is $\upsilon= dd^c \tau >0$.
On $\mathbb{C}^n\setminus\{0\}$, we further define
$\omega = dd^c \log \tau \ge 0$ and $\sigma = d^c\log\tau \wedge \omega^{n-1}$,
where $n=\dim(\mathbb{C}^n)$ (see \cite[pp. 6]{Stoll-1974}).

\smallskip
Let $G\neq \varnothing$ be an open subset of $\mathbb{C}^n$. Let $f$ be a holomorphic function in $\mathbb{C}^n$.
Take $a\in G$. Let $G_a$ be the connectivity component of $G$ containing $a$. Assume $f\mid_{G_a}\not\equiv 0$. Then $f$ admits a local expansion
\begin{align*}
f(z)=\sum_{\lambda=p}^{\infty}P_\lambda(z-a),
\end{align*}
where $P_\lambda$ is homogeneous of degree $\lambda$ and $P_p\not\equiv0$. The polynomials $P_{\lambda}$ depend on $f$ and $a$ only.
The integer $\mu_f^0(a)=p$ is called the zero multiplicity of $f$ at $a$ (see \cite[pp. 12]{Stoll-1974}).

\smallskip
Let $f$ be a meromorphic function on $G$, where $G\neq \varnothing$ is an open subset of $\mathbb{C}^n$.
Take $a\in G$ and $c\in\mathbb{P}^1$. Let $G_a$ be the component of $G$ containing $a$. If $0\equiv f\mid_{G_a}\not\equiv c$, define $\mu^c_f(a)=0$. Assume $0\not\equiv f\mid_{G_a}\not\equiv c$. Then an open connected neighborhood $U$ of $a$ in $G$ and holomorphic functions $g\not\equiv 0$ and $h\not\equiv 0$ exist on $U$ such that $h. f\mid_U=g$ and $\dim g^{-1}(0)\cap h^{-1}(0)\leq n-2$, where $n=\dim(\mathbb{C}^n)$. Therefore the $c$-multiplicity of $f$ is just $\mu^c_f=\mu^0_{g-ch}$ if $c\in\mathbb{C}$ and $\mu^c_f=\mu^0_h$ if $c=\infty$. The function $\mu^c_f:G\to \mathbb{Z}$ is nonnegative and is called the $c$-divisor of $f$ (see \cite[pp. 12]{Stoll-1974}).
If $f\not\equiv0$ on each component of $G$, the divisor of $f$ is $\mu_f=\mu_f^0-\mu_f^\infty$.
The function $f$ is holomorphic if and only if $\mu_f\ge0$.

\smallskip
A function $\nu:G\to \mathbb{Z}$ is said to be a divisor if and only if for each $a\in G$ an open, connected neighborhood $U$ of $a$ in $G$ and a meromorphic function $f\not\equiv 0$ exist such that $\nu\mid_{U}=\mu_f$. A divisor $\nu:G\to \mathbb{Z}$ is non-negative if and only if for each $a\in G$ an open, connected neighborhood $U$ of $a$ in $G$ and a holomorphic function $f\not\equiv 0$ exist such that $\nu\mid_{U}=\mu_f$ (see \cite[pp. 13]{Stoll-1974}).
We denote $\operatorname{supp}\nu=\overline{\{z\in G:\nu(z)\ne0\}}$.

\smallskip
Take $0<R\leq +\infty$. Let $\nu$ be a divisor on $\mathbb{C}^n(R)$ with $A=\displaystyle \operatorname{supp}\nu$. For $t>0$, the counting function $n_{\nu}$ is defined by
\begin{align*}
 \displaystyle n_{\nu}(t)=t^{-2(m-1)}\int_{A[t]}\nu\; \upsilon^{n-1},
 \end{align*}
where $n=\dim(\mathbb{C}^n)$. We know that
\begin{align*}
 \displaystyle n_{\nu}(t)=\int_{A(t)}\nu\;\omega^{n-1}+n_{\nu}(0).
 \end{align*}

For $0<s<r<R$, we define the valence function of $\nu$ by
\begin{align*}
N_{\nu}(r)=N_{\nu}(r,s)=\int_{s}^r n_{\nu}(t)\frac{dt}{t}.
\end{align*}

For $\nu=\mu_f^a$, we write $n(t,a;f)$ and $N(r,a;f)$, with the usual conventions for $a=\infty$.

A non-negative divisor $\nu:\mathbb{C}^n\to \mathbb{Z}_+$ is said to be algebraic if and only if $\nu$ is  the zero divisor of a polynomial. Thus a divisor $\nu:\mathbb{C}^n\to \mathbb{Z}_+$ is algebraic if and only if $n_{\nu}$ is bounded, which implies that $N_{\nu}=O(\log r)$ (see \cite[pp. 19]{Stoll-1974}).

\smallskip
Let $G\neq \varnothing$ be an open subset in $\mathbb{C}^n$. Let $f$ be a meromorphic function in $G$ in the sense that $f$ can be written as a quotient of two relatively prime holomorphic functions. We will write $f = (f_0, f_1)$ where $f_0 \not\equiv 0$. The standard definition of the Nevanlinna characteristic function of $f$ is given by (see \cite[pp. 16-17]{Stoll-1974})
\begin{align*}
T_f(r,s) := \int_s^r \frac{A_f(t)}{t}\, dt,
\end{align*}
where $0 < s < r$ and
\begin{align*}
A_f(t)
= \frac{1}{t^{2n-2}} \int_{\mathbb{C}^n(t)} f^*(\ddot{\omega}) \wedge \upsilon^{n-1}
= \int_{\mathbb{C}^n(t)} f^*(\ddot{\omega}) \wedge \omega^{\,n-1} + A_f(0),
\end{align*}
where $n=\dim(\mathbb{C}^n)$.
Here the pullback $f^*(\ddot{\omega})$ satisfies
\begin{align*}
f^*(\ddot{\omega}) = dd^c \log\left( |f_0|^2 + |f_1|^2 \right)
\end{align*}
for all $z$ outside of the set of indeterminacy $I_f:= \{ z \in \mathbb{C}^n : f_0(z) = f_1(z) = 0 \}$ of $f$.

Take $a\in \mathbb{P}^1$ and $0<R\leq +\infty$. Let $f\not\equiv 0$ be a meromorphic function on $\mathbb{C}^n(R)$.
For $0<r<R$, define the compensation of $f$ for $a$ by
\begin{align*}
m^a_f(r)=\int\limits_{\mathbb{C}^n\langle r\rangle}\log \frac{1}{||f,a||} \;\sigma,
\end{align*}
where $||f,a||$ denotes the chordal distance from $f$ to $a\in\mathbb{P}^1$.
Then the First Main Theorem of Nevanlinna theory becomes
\begin{align*}
T_f(r)=T_f(r,s)=N_{\mu^a_f}(r,s)+m^a_f(r)-m^a_f(s),
\end{align*}
where $0<s<r$.

There is slightly different way to continue the formulation of Nevanlinna theory from here (see \cite[pp.15]{Hu-Li-Yang-2003}).
Take $0<R\leq +\infty$. Let $f\not\equiv 0$ be a meromorphic function on $\mathbb{C}^n(R)$. Let $0<s<r<R$. 
Now with the help of the positive logarithm function, we define the proximity function of $f$ by
\begin{align*}
\displaystyle m(r, f)=\int_{\mathbb{C}^n\langle r\rangle} \log^+ |f|\;\sigma \geq  0.
\end{align*}

The characteristic function of $f$ is defined by $T(r,f)=m(r,f)+N(r,f)$. We know that (see \cite[pp.15]{Hu-Li-Yang-2003})
\begin{align*}
\displaystyle T\left(r,\frac{1}{f}\right)=T(r,f)-\int_{\mathbb{C}^n\langle s\rangle} \log |f|\;\sigma.
\end{align*}

We define $m(r,a;f)=m(r,f)$ if $a=\infty$ and $m(r,a;f)=m\big(r,\frac{1}{f-a}\big)$ if $a$ is finite complex number. Now if $a\in\mathbb{C}$, then the first main theorem of Nevanlinna theory becomes $m(r,a;f)+N(r,a;f)=T(r,f) + O(1)$, where $O(1)$ denotes a bounded function when $r$ is sufficiently large.

Finally, if we compare the functions $T_f(r)$ and $T(r,f)$, then we have (see \cite[pp.19]{Hu-Li-Yang-2003})
\begin{align*}
T_f(r)=T(r,f)+O(1).
\end{align*}

Clearly $f$ is non-constant, then $T(r, f) \rightarrow \infty$ as $r \rightarrow$ $\infty$. Further $f$ is rational if and only if $T(r,f)=O(\log r)$ (see \cite[pp. 19]{Stoll-1974}).
On the other hand, if $f$ is transcendental, then 
\[\lim\limits_{r \rightarrow \infty} \frac{T(r, f)}{\log r}=+\infty.\]

\section{{\bf Auxiliary lemmas}}
Let $U\subset \mathbb{C}^n$ be an open subset. A closed subset $A\subset U$ is said to be analytic, 
if for each $a\in A$, there are a finite number of holomorphic functions $f_1(z),f_2(z),\ldots,f_l(z)$ defined in a neighbourhood $N(a)$ of $a$ such that (see \cite[pp. 42]{NW})
\begin{align*}
\displaystyle A\cap N(a)=\{z\in N(a):f_1(z)=f_2(z)=\cdots=f_l(z)=0\}.
\end{align*}

\smallskip
If $f_j(z),1\leq j\leq l$ can be taken so that their differentials at $a\in A$, $df_1(a), df_2(a), \ldots,df_l(a)$ are linearly independent, then $a$ is called a regular point of $A$. A point of $A$ which is not regular called a singular point. The subset of all singular points of $A$ is denoted by $S(A)$. Set $R(A)=A\backslash  S(A)$. If $S(A)=\varnothing$, then $A$ is said to be regular. 

\smallskip
Let $R(A)=\bigcup_{\lambda}A_{\lambda}'$ be the decomposition into the connected components. Then the closure $A_{\lambda}=\ol{A'_{\lambda}}$ is analytic and $A=\bigcup_{\lambda}A_{\lambda}$ is a locally finite covering. Set $\dim(A)=\max_a \dim_a(A)$, where $\dim_a(A)=\max_{a\in A_{\lambda}}\dim(A_{\lambda})$. If $\dim_a(A)=\dim(A)$ at all points $a\in A$, then $A$ is said to be of pure dimension $\dim(A)$. 

Let $U$ be an open subset of $\mathbb{C}^n$. Let $A \subset U$ be an analytic subset of pure dimension $l$ and let  $\iota_{R(A)} : R(A)\longrightarrow \mathbb{C}^n$ be the inclusion map. For a compact subset $K\subset U$ the integral
\begin{align*}
 \displaystyle \int_{K\cap A} \alpha ^l=\int_{K\cap R(A)} \alpha^l = \int_{K\cap R(A)} {\iota}^{\ast} _{R(A)} \alpha^l
 \end{align*}
is considered as a measure of $K\cap A$ (see \cite[p. 45]{NW}). We now recall the following lemma.

\begin{lem}\cite[Lemma 2.2.9]{NW}\label{Lm0} Let the notation be as above.
\begin{enumerate}
\item[(i)] If $l=\dim A < n$, then Lebesgue measure of $A$ in $\mathbb{C}^n$ is zero.
\item [(ii)] $\int_{K\cap R(A)} \alpha^l < \infty$.
\end{enumerate}
\end{lem}

\begin{lem}\label{Lm1}\cite[pp. 78]{Clunie-1970} Let $f$ be a meromorphic function in $\mathbb{C}$ and let $g$ be an entire function in $\mathbb{C}$. Suppose $f$ and $g$ are transcendental. Then 
\begin{align*}
\limsup\limits_{r\to\infty}\frac{T(r,f(g))}{T(r,f)}=\infty.
\end{align*}
\end{lem}

\begin{lem}\label{Lm2} Let $f$ be a meromorphic function in $\mathbb{C}$ and let $g:\mathbb{C}^n\to \mathbb{C}$ be an entire function. Suppose $f$ and $g$ are transcendental. Then 
\begin{align*}
\limsup\limits_{r\to\infty}\frac{T(r,f(g))}{T(r,f)}=\infty.
\end{align*}
\end{lem}
\begin{proof} By the given condition $g$ is transcendental. Then $g$ has the Taylor expansion near origin,
\begin{align*}
g(\zeta)=\sum\limits_{k=0}^{\infty} P_k(\zeta),
\end{align*}
where $P_k$ is a homogeneous polynomial of degree $k$ and there are infinitely many $k$'s 
such that $P_k\not\equiv 0$. Denote by $I$ the set of such $k$'s. Then the sets
\begin{align*}
Z_k=\left\lbrace \zeta\in\mathbb{C}^n: P_k(\zeta)=0\right\rbrace, \quad k\in I
\end{align*}
are algebraic. We know that $Z_k$ is an analytic set of pure dimension $n-1$ and so by Lemma \ref{Lm0}, the Lebesgue measure of $Z_k$ in $\mathbb{C}^n$ is zero for $k=1,2,\ldots$. Therefore $\bigcup_{k\in I} Z_k$ has also zero Lebesgue measure. Let us take $\zeta\in \mathbb{C}^n\backslash \bigcup\limits_{k\in I} Z_k$ such that $\zeta\neq 0$. Set $\xi=\frac{\zeta}{||\zeta||}$. Then
\begin{align*}
P_k(\xi)=\frac{1}{||\zeta||^k}P_k(\zeta)\neq 0,\quad k\in I.
\end{align*}

Define a holomorphic mapping $j_{\xi}:\mathbb{C}\to \mathbb{C}^n$ by $j_{\xi}(z)=z\xi$. Clearly
\begin{align*}
g_{\xi}(z)=g\circ j_{xi}(z)=g(z\xi)=\sum\limits_{k=0}^{\infty} P_k(\xi)z^k
\end{align*}
is a transcendental entire function on $\mathbb{C}$. Set $F=f\circ g=f(g)$. Note that $F_{\xi}=f\circ g_{\xi}=f(g_{\xi})$.
On the other hand we know that (see \cite[pp.19]{Stoll-1974})
\begin{align}\label{EqL2.1} 
T_{F_{\xi}}(r,0)\leq \frac{1+\theta}{(1-\theta)^{2n-1}} T_{F}\left(\frac{r}{\theta},0\right)
\end{align}
where $0<\theta<1$. Clearly from (\ref{EqL2.1}), we get
\begin{align}\label{EqL2.2} 
\limsup\limits_{r\to \infty} \frac{T_{F_{\xi}}(r,0)}{T_f(r,0)}\leq \frac{1+\theta}{(1-\theta)^{2n-1}} \limsup\limits_{r\to\infty}\frac{T_{F}\left(\frac{r}{\theta},0\right)}{T_f(r,0)}.
\end{align}
According to Lemma \ref{Lm1}, we have
\begin{align*}
\limsup\limits_{r\to \infty} \frac{T(r,f(g_{\xi}))}{T(r,f)}=\limsup\limits_{r\to \infty} \frac{T_{F_{\xi}}(r,0)}{T_f(r,0)}=\infty.
\end{align*}
Consequently from (\ref{EqL2.2}), we obtain
\begin{align*}
\limsup\limits_{r\to \infty} \frac{T(r,f(g))}{T(r,f)}=\limsup\limits_{r\to \infty} \frac{T_{F}(r,0)}{T_f(r,0)}=\infty.
\end{align*}
\end{proof}

\begin{lem}\label{L1}\cite{Nevanlinna1936}
For any given $\varepsilon$ with $0<\varepsilon< \frac{1}{2}$, we have for $r>r_0$
\begin{align*}
	N(r,a;f)\geq T(r,f)-2\left\lbrack T(r,f)\right\rbrack^{(1+\varepsilon)/2}
\end{align*}	
for all values a except those which belong to a set $E(r_0,\varepsilon)$ of finite logarithmic capacity.
\end{lem}

\begin{lem}\label{L1.1} Let $f:\mathbb{C}\to \mathbb{P}^1$ be a meromorphic function. If $g:\mathbb{C}^n\to \mathbb{C}$ is a non-constant polynomial of degree $m$, then 
	\begin{align*}
		T(r,f(g))\geq (1-o(1))T(r^m,f).
	\end{align*}
	\end{lem}
	
	\begin{proof}
	Let $c\in\mathbb{C}$ be such that $f(z)-c$ has infinitely many zeros $c_1,c_2,\ldots,c_j,\ldots$. Then $g(z)=c_j$ implies that $f(g)=c$. Let
	\begin{align*}
		A=\operatorname{supp} \mu^c_{f(g)},\quad B_j=\operatorname{supp} \mu^{c_j}_{g} \quad \text{and} \quad C=\operatorname{supp} \mu^c_f, 
	\end{align*}
	where $j=1,2,\ldots$. Note that in some neighbourhood of $c_j$, we can write 
	\begin{align*}
		f(z)-c=(z-c_j)^{m_j} \phi_j(z),
		\end{align*}
		where $\phi_j(z)$ is analytic at $c_j$ and $\phi(c_j)\neq 0$ and $m_j\in\mathbb{N}$. Clearly
		\begin{align}\label{l1}
			f(g(z))-c=(g(z)-c_j)^{m_j} \phi_j(g(z)),
		\end{align}
		where $j=1,2,\ldots$.
		
		We can choose $t_0$ so large such that for $c_j\in C[t^m]$ atleast $m$ roots of $g(z)=c_j$ satisfy $B_j\left\lbrack t\right\rbrack$
		if $t>t_0$, where $j=1,2,\ldots,$. Then in view of (\ref{l1}), we obtain
		\begin{align*}
			n_{\mu^c_{f(g)}}(t)=t^{-2(n-1)} \int_{A[t]}\mu^{c}_{f(g)}\nu_n^{n-1}\geq t^{2m(n-1)}\int_{C[t^m]}m\mu^c_f\nu_n^{n-1}=n_{\mu^c_f}(t^m).
		\end{align*}
		Therefore for $r>r_0$, we have 
		\begin{align*}
			N(r,c;f)=\int\limits^r_{r_0} \frac{n_{\mu^c_{f(g)}}(t)}{t} dt\geq \int\limits^r_{r_0}\frac{m n_{\mu^c_f}(t^m)}{t} dt.
		\end{align*}
		Making the substitution of variable $s=t^m$, we obtain
		\begin{align}\label{l2}
			N(r,c;f(g))\geq \int\limits_{r_0^m}^{r^m} \frac{n_{\mu^c_f}(s)}{s} ds=N(r^m,c;f)+O(1).
		\end{align}
		According to Lemma \ref{L1}, we can actually choose $c$ such that
		\begin{align}\label{l3}
			N(r,c;f)\geq (1-o(1)) T(r,f)\quad r\to \infty.
		\end{align}
		Now from (\ref{l2}) and (\ref{l3}), we deduce that
		\begin{align*}
			T(r,f(g))\geq N(r,c;f(g))\geq (1-o(1)) T\left(r^m,f\right)\quad \text{as}\;\; r\to \infty.
		\end{align*}
		
	\end{proof}

\begin{lem}\label{L2} \cite[Lemma 1.37]{Hu-Li-Yang-2003} Let $f:\mathbb{C}^n\to\mathbb{P}^1$ be a non-constant meromorphic function and $I=(\alpha_1,\ldots,\alpha_n)\in \mathbb{Z}^n_+$ be a multi-index. Then for any $\varepsilon>0$, we have
\begin{align*}
m\left(r,\frac{\partial^I(f)}{f}\right)\leq |I|\log^+T(r,f)+|I|(1+\varepsilon)\log^+\log T(r,f)+O(1)
\end{align*}
for all large $r$ outside a set $E$ with $\int_E d\log r<\infty$.
\end{lem}

\begin{lem}\label{L3} \cite[Theorem 1.26]{Hu-Li-Yang-2003} Let $f:\mathbb{C}^n\to\mathbb{P}^1$ be a non-constant meromorphic function. Assume that 
$R(z, w)=\frac{A(z, w)}{B(z, w)}$. Then
\begin{align*}
 T\left(r, R_f\right)=\max \{p, q\} T(r, f)+O\Big(\sideset{}{_{j=0}^{p}}{\sum}T(r, a_j)+\sideset{}{_{j=0}^{q}}{\sum}T(r, b_j)\Big),
 \end{align*}
for all large $r$ outside a set $E$ with $\int_E d\log r<\infty$, where $R_f(z)=R(z, f(z))$ and two coprime polynomials $A(z, w)$ and $B(z,w)$ are given
respectively $A(z,w)=\sum_{j=0}^p a_j(z)w^j$ and $B(z,w)=\sum_{j=0}^q b_j(z)w^j$.
\end{lem}

\section{{Proof of Theorem \ref{t1.1}}}
\begin{proof} From (\ref{Eq}), we have
\begin{align}\label{EqT1.0}
\frac{\partial h(z)}{\partial z_i}=aG^g_h(z)+bh(z)+c
\end{align}
where $z=(z_1,z_2,\ldots,z_n)$ and $i\in\{1,2,\ldots,n\}$. 
Note that
\begin{align}\label{EqT1.1}
N\left(r,\frac{\partial h(z)}{\partial z_i}\right)\leq 2N(r,h(z)),
\end{align}
where $z=(z_1,z_2,\ldots,z_n)$ and $i\in\{1,2,\ldots,n\}$.
Now using Lemma \ref{L2} and (\ref{EqT1.1}) to (\ref{EqT1.0}), we deduce that
\begin{align}\label{EqT1.2}
T\left(r,\frac{\partial h(z)}{\partial z_i}-bh(z)\right)&=N\left(r,\frac{\partial h(z)}{\partial z_i}-bh(z)\right)+m\left(r,\frac{\partial h(z)}{\partial z_i}-hf(z)\right)\\&\leq 
2N(r,h(z))+m\left(r,\frac{\frac{\partial h(z)}{\partial z_i}-bh(z)}{h(z)}\right)+m(r,h)\nonumber\\&\leq
2T(r,h)+o(T(r,h))\nonumber
\end{align}
outside possibly a set of finite logarithmic measure, where $i\in\{1,2,\ldots,n\}$.

Therefore using Lemma \ref{L3} and (\ref{EqT1.2}) to (\ref{EqT1.0}), we obtain
\begin{align}\label{EqT1.3}
T(r,G^g_h)=T\left(r, \frac{\frac{\partial h(z)}{\partial z_i}-bh(z)}{a}-\frac{c}{a}\right)&\leq T\left(r,\frac{\partial h(z)}{\partial z_i}-bh(z)\right)+O(1)\\&\leq 
2T(r,h)+o(T(r,h))\nonumber
\end{align}
outside possibly a set of finite logarithmic measure, where $i\in\{1,2,\ldots,n\}$.

Now we consider the following two cases.\par

\medskip
{\bf Case 1.} Let $g$ be a transcendental entire function in $\mathbb{C}^n$. Now from (\ref{EqT1.0}), we have 
\begin{align}\label{EqT1.4}
\frac{\partial h(z)}{\partial z_i}-bh(z)-c=aG^g_h(z).
\end{align}

\medskip
First we suppose that $f$ is a rational function in $\mathbb{C}$. Clearly $h$ is a rational function in $\mathbb{C}^n$. Let $h=\frac{P}{Q}$, where $P$ and $Q$ are rational functions in $\mathbb{C}^n$ of degree $p$ and $q$ respectively. Then by Lemma \ref{L3}, we have 
\begin{align*}
T(r, G^g_h)=\max\{p,q\}\;T(r,g)+O(1)\neq O(\log r),
\end{align*}
which shows that $G^g_h$ is transcendental. It is easy to conclude from (\ref{EqT1.4}) that $h$ is a rational function in $\mathbb{C}^n$ and so $f$ is a rational function in $\mathbb{C}$, which is impossible.

\medskip
Next we suppose that $f$ is a transcendental meromorphic function in $\mathbb{C}$. Consequently $h$ is also a transcendental meromorphic function in $\mathbb{C}^n$. Then from (\ref{EqT1.3}), we obtain 
\begin{align*}
\limsup\limits_{r\not\in E, r\to\infty}\frac{T\left(r,G^g_h\right)}{T(r,h)}<\infty,
\end{align*}
which is impossible by Lemma \ref{Lm2}.

\medskip
{\bf Case 2.} Let $g$ be a polynomial in $\mathbb{C}^n$. Suppose $\deg(g)=m$. If $m =1$, then
the conclusions $(i)$ and $(ii)$ of Theorem \ref{t1.1} both hold. Next suppose $m\geq 2$. Then by Lemma \ref{L1.1}, we get
\begin{align}\label{EqT1.5}
T\left(r,G^g_h\right)\geq (1-\varepsilon_1)T(r^m,h)
\end{align}
for any $\varepsilon_1>0$. We know that $T(r, h)$ is a convex function of $\log r$. Thus, for large $r$, we have 
\begin{align*}
\frac{T(r,h)-T(1,h)}{\log r-\log 1}\leq \frac{T(r^m,h)-T(1,h)}{\log r^m-\log 1},
\end{align*}
which shows that 
\begin{align}\label{EqT1.6}
T(r,h)\leq \frac{1}{m}(1+\varepsilon_2)T\left(r^m,h\right)
\end{align}
for any $\varepsilon_2>0$ and large $r$. Now using (\ref{EqT1.3}), (\ref{EqT1.5}) to (\ref{EqT1.6}), we obtain
\begin{align}\label{EqT1.7}
mT(r,h)\leq \frac{1+\varepsilon_2}{1-\varepsilon_1}\left(2T(r,h)+o(T(r,h)\right)
\end{align} 
for large $r$ possibly outside a set of finite logarithmic measure. Since $m\geq 2$,  from (\ref{EqT1.7}), we deduce that $m= 2$.

We now prove that $f$ has at most one pole. For the sake of simplicity, we suppose that $f$ has two poles at $a_1$ and $a_2$ of multiplicities $m_1$ and $m_2$ respectively. Assume that 
\begin{align}\label{EqT1.8}
f(z)=\frac{A(z)}{(z-a_1)^{m_1}(z-a_2)^{m_2}},
\end{align}
where $A(z)$ is a meromorphic function, which is analytic at $a_1$, $a_2$ and $A(a_1)\neq 0$ and $A(a_2)\neq 0$. Now from (\ref{EqT1.8}), we have 
\begin{align}\label{EqT1.9}
h(z)=\frac{A(z_1+z_2+\ldots+z_n)}{(z_1+z_2+\ldots+z_n-a_1)^{m_1}(z_1+z_2+\ldots+z_n-a_2)^{m_2}},
\end{align}
where $z=(z_1,z_2,\ldots,z_n)$. Now using (\ref{EqT1.9}) to (\ref{EqT1.0}), we get
\begin{align}\label{EqT1.10}
&\frac{B_i(z_1+\ldots+z_n)}{(z_1+\ldots+z_n-a_1)^{m_1+1}(z_1+\ldots+z_n-a_2)^{m_2+1}}\\&-\frac{A(z_1+\ldots+z_n)}{(z_1+\ldots+z_n-a_1)^{m_1}(z_1+\ldots+z_n-a_2)^{m_2}}\nonumber\\&=
a\frac{A(ng(z))}{(ng(z)-a_1)^{m_1}(ng(z)-a_2)^{m_2}}+c,\nonumber
\end{align}
where $i\in\{1,2,\ldots,n\}$ and
\begin{align*}
B_i(z_1+\ldots+z_n)&=(z_1+\ldots+z_n-a_1)(z_1+\ldots+z_n-a_2)\frac{\partial A(z_1+\ldots+z_n)}{\partial z_i}\\&-m_1(z_1+\ldots+z_n-a_2)A(z_1+\ldots+z_n)\\&-m_2(z_1+\ldots+z_n-a_1)A(z_1+\ldots+z_n)
\end{align*}
for $i\in\{1,2,\ldots,n\}$. Clearly $B_i(a_1)\neq 0$ and $B_i(a_2)\neq 0$ for $i\in\{1,2,\ldots,n\}$.

Let $\hat z\in\mathbb{C}^n$ be a zero of $z_1+\ldots+z_n-a_1$. Then from (\ref{EqT1.10}), we can say that $\hat z$ is a zero of $ng(z)-a_1$ or $ng(z)-a_2$. For the sake of simplicity, we may assume that $\hat z$ is a zero of $ng(z)-a_1$. If $\hat z$ is a simple zero of $ng(z)-a_1$, then from (\ref{EqT1.10}), we get a contradiction. Hence $\hat z$ is a $ng(z)-a_1$ of multiplicity atleast $2$. Since $\deg(g)=2$. it follows that $\hat z$ is a zero $ng(z)-a_1$ of multiplicity exactly $2$. Then we can write 
\begin{align}\label{EqT1.11}
ng(z)-a_1=\tilde a\left(z_1+\ldots+z_n-a_1\right)^2
\end{align}
where $\tilde a$ is a non-zero constant in $\mathbb{C}$.

Let $\tilde z\in\mathbb{C}^n$ be a zero of $z_1+\ldots+z_n-a_2$. Certainly from (\ref{EqT1.10}), we can say that $\tilde z$ is a zero of $ng(z)-a_2$. If $\hat z$ is a simple zero of $ng(z)-a_1$, then from (\ref{EqT1.10}), we get a contradiction. Hence $\hat z$ is a $ng(z)-a_2$ of multiplicity atleast $2$. Since $\deg(g)=2$. it follows that $\hat z$ is a zero $ng(z)-a_2$ of multiplicity exactly $2$. Then we can write 
\begin{align}\label{EqT1.12}
ng(z)-a_2=\tilde b\left(z_1+\ldots+z_n-a_2\right)^2,
\end{align}
where $\tilde b$ is a non-zero constant in $\mathbb{C}$.

Now from (\ref{EqT1.11}) and (\ref{EqT1.12}), we deduce that
\begin{align*}
\tilde a \left(z_1+\ldots+z_n-a_1\right)^2+a_1\equiv \tilde b\left(z_1+\ldots+z_n-a_2\right)^2+a_2,
\end{align*}
which is impossible here. We have thus proved that $f$ has at most one pole. Therefore, we can write
\begin{align}\label{EqT1.12a}
f(z)=h_1(z)h_2(z),
\end{align}
where $h_1(z)=\frac{1}{(z-w_0)^l}$, $z\in\mathbb{C}$, $l\in\mathbb{N}\cup\{0\}$ and $h_2(z)$ is an entire function in $\mathbb{C}$ with
$h_2(w_0)\neq 0$. Consequently
\begin{align}\label{EqT1.13}
h(z)=h_1(z_1+z_2+\ldots+z_n)h_2(z_1+z_2+\ldots+z_n),
\end{align}
where $z=(z_1,z_2,\ldots,z_n)$.

Now we consider the following two sub-cases.

\medskip
{\bf Sub-case 2.1.} Let $h_2$ be a transcendental entire function in $\mathbb{C}$. Using (\ref{EqT1.13}) to (\ref{EqT1.0}), we obtain
\begin{align}\label{EqT1.13a}
&\frac{\partial h_1(z_1+\ldots+z_n)}{\partial z_i}h_2(z_1+\ldots+z_n)+h_1(z_1+\ldots+z_n) \frac{\partial h_2(z_1+\ldots+z_n)}{\partial z_i}\\&=aG^g_{h_1}(z)G^g_{h_2}(z)+bh_1(z_1+\ldots+z_n)h_2(z_1+\ldots+z_n)+c,\nonumber
\end{align}
i.e.,
\begin{align}\label{EqT1.14}
&G^g_{h_2}(z)\\&=\frac{\frac{\partial h_1(z_1+\ldots+z_n)}{\partial z_i}h_2(z_1+\ldots+z_n)+h_1(z_1+\ldots+z_n) \frac{\partial h_2(z_1+\ldots+z_n)}{\partial z_i}}{G^g_{h_1}(z)}\nonumber\\&-\frac{bh_1(z_1+\ldots+z_n)h_2(z_1+\ldots+z_n)+c}{G^g_{h_1}(z)}\nonumber\\&=
h_2(z_1+\ldots+z_n)\left(\frac{\frac{\partial h_1(z_1+\ldots+z_n)}{\partial z_i}+h_1(z_1+\ldots+z_n) \frac{\frac{\partial h_2(z_1+\ldots+z_n)}{\partial z_i}}{h_2(z_1+\ldots+z_n)}-bh_1(z_1+\ldots+z_n)}{G^g_{h_1}(z)}\right)\nonumber\\&-\frac{c}{G^g_{h_1}(z)},\nonumber
\end{align}
where $z=(z_1,z_2,\ldots,z_n)$ and $i\in\{1,2,\ldots,n\}$. Note that $g(z)$ is a polynomial in $\mathbb{C}^n$ of degree $2$ and $h_1(z_1+\ldots+z_n)=\frac{1}{(z_1+\ldots+z_n-w_0)^l}$. Now applying Lemma \ref{L2} to (\ref{EqT1.14}), we deduce that
\begin{align}\label{EqT1.15}
T\left(r,G^g_{h_2}\right)&=m\left(r,G^g_{h_2}\right)\\&\leq m\left(r,h_2(z_1+\ldots+z_n)\right)+m\left(r,\frac{\frac{\partial h_2(z_1+\ldots+z_n)}{\partial z_i}}{h_2(z_1+\ldots+z_n)}\right)+O(\log r)\nonumber\\&
\leq T(r,h_2(z_1+\ldots+z_n))+o(T(r,h_2(z_1+\ldots+z_n)))
\end{align}
possibly outside a set of finite logarithmic measure. On the other hand by Lemma \ref{L1.1}, we have
\begin{align}\label{EqT1.16}
T\left(r,G^g_{h_2}\right)\geq (1-\varepsilon_3)T\left(r^2,h_2(z_1+\ldots+z_n)\right)
\end{align}
for any $\varepsilon_3>0$ and large $r$. We know that $T(r,h_2(z_1+\ldots+z_n))$ is an increasing convex function of $\log r$ so that$T(r,h_2(z_1+\ldots+z_n))/\log r$ is finally increasing and hence
\begin{align}\label{EqT1.17}
T(r,h_2(z_1+\ldots+z_n))\leq \frac{1}{2} T\left(r^2,h_2(z_1+\ldots+z_n)\right)
\end{align}
for sufficiently large $r$. Now from (\ref{EqT1.15})-(\ref{EqT1.17}), we obtain
\begin{align*}
T(r,h_2(z_1+\ldots+z_n))\leq \frac{1}{2}\frac{1}{1-\varepsilon_3}\left(T(r,h_2(z_1+\ldots+z_n))+o(T(r,h_2(z_1+\ldots+z_n))\right),
\end{align*}
which shows that $T(r,h_2(z_1+\ldots+z_n))=o(T(r,h_2(z_1+\ldots+z_n)))$. Therefore we get a contradiction.

\medskip
{\bf Sub-case 2.2.} Let $h_2$ be a polynomial in $\mathbb{C}$. Note that $g(z)$ is a polynomial in $\mathbb{C}^n$ of degree $2$ and $h_1(z_1+\ldots+z_n)=\frac{1}{(z_1+\ldots+z_n-w_0)^l}$. Now from (\ref{EqT1.13a}), we get
\begin{align}\label{EqT1.18}
&\frac{-lh_2(z_1+\ldots+z_n)}{(z_1+\ldots+z_n-w_0)^{l+1}}+\frac{1}{(z_1+\ldots+z_n-w_0)^l}\frac{\partial h_2(z_1+\ldots+z_n)}{\partial z_i}\\&=a\frac{h_2(ng(z))}{(ng(z)-w_0)^l}+b\frac{1}{(z_1+\ldots+z_n-w_0)^l}h_2(z_1+\ldots+z_n)+c,\nonumber
\end{align}
where $z=(z_1,z_2,\ldots,z_n)$ and $i\in\{1,2,\ldots,n\}$. Let $z^0=(z^0_1,z^0_2,\ldots,z^0_n)\in\mathbb{C}^n$ be a zero of $z_1+\ldots+z_n-w_0$. Since $h_2(w_0)\neq 0$, it follows that $h_2(z^0_1+\ldots+z^0_n)=h_2(w_0)\neq 0$ and $h_2(ng(z^0))=h_2(w_0)\neq 0$.
Now from (\ref{EqT1.18}), we can say that $\hat z$ is a zero of $ng(z)-w_0$. If $z^0$ is a simple zero of $ng(z)-w_0$, then from (\ref{EqT1.18}), we get a contradiction. Hence $z^0$ is a $ng(z)-a_1$ of multiplicity atleast $2$. Since $\deg(g)=2$. it follows that $z^0$ is a zero $ng(z)-w_0$ of multiplicity exactly $2$. Then we can write 
\begin{align}\label{EqT1.19}
ng(z)-w_0=\tilde a\left(z_1+\ldots+z_n-w_0\right)^2,
\end{align}
where $\tilde a$ is a non-zero constant in $\mathbb{C}$.

Now using (\ref{EqT1.19}) to (\ref{EqT1.18}) and then comparing the order of the pole at $z^0$ in the two sides of (\ref{EqT1.18}), we deduce that $l+1=2l$, i.e., $l=1$. Therefore from (\ref{EqT1.12a}), we get
\begin{align}\label{EqT1.19a}
f(z)=\frac{h_2(z)}{z-w_0},
\end{align}
where $z\in\mathbb{C}$. Also from (\ref{EqT1.18}), we have
\begin{align}\label{EqT1.20}
&\left(-lh_2(z_1+\ldots+z_n)+(z_1+\ldots+z_n-w_0)\frac{\partial h_2(z_1+\ldots+z_n)}{\partial z_i}\right)(ng(z)-w_0)\\&=ah_2(ng(z))(z_1+\ldots+z_n-w_0)^2+b(z_1+\ldots+z_n-w_0)h_2(z_1+\ldots+z_n)(ng(z)-w_0)\nonumber\\&+c(z_1+\ldots+z_n-w_0)(ng(z)-w_0),\nonumber
\end{align}
where $z=(z_1,z_2,\ldots,z_n)$ and $i\in\{1,2,\ldots,n\}$. 

Let $\deg(h_2)=d$. If $d\geq 2$, then it is easy to check that
the left hand side of (\ref{EqT1.20}) has degree $d+2$, while the right hand side has degree $2d + 2$, which is impossible. Thus, we have that $d = 0$ or $d = 1$. Therefore from (\ref{EqT1.19a}), we can write 
\begin{align}\label{EqT1.21}
f(z)=\beta+\frac{\alpha}{z-w_0},
\end{align}
for some constants $\alpha\neq 0, \beta$, where $z\in\mathbb{C}$. Now using (\ref{EqT1.21}) to (\ref{EqT1.0}), we get 
\begin{align*}
-\frac{\alpha}{(z_1+\ldots+z_n)^2}=\frac{a\alpha}{g(z)-w_0}+\frac{b\alpha}{z_1+\ldots+z_n-w_0}+c+(a+b)\beta,
\end{align*}
which shows that
\begin{align}\label{EqT1.22}
g(z)-w_0=\frac{a\alpha(z_1+\ldots+z_n-w_0)^2}{-\alpha-b\alpha(z_1+\ldots+z_n-w_0)-((a+b)\beta+c)(z_1+\ldots+z_n-w_0)^2},
\end{align}
where $z\in\mathbb{C}^n$.

Since $\deg(g)=2$, from (\ref{EqT1.22}), we conclude that $b\alpha=(a+b)\beta+c=0$, which shows that $b=0$ and $a\beta+c=0$. 
Consequently from (\ref{EqT1.22}), we have
\begin{align*}
ng(z)=w_0-a\left(z_1+\ldots+z_n-w_0\right)^2
\end{align*}

Therefore
\begin{align*}
f(z)=\beta+\frac{\alpha}{z-w_0}\quad \text{and}\quad ng(z)=w_0-a\left(z_1+\ldots+z_n-w_0\right)^2,
\end{align*}
$b=0$ and $\alpha\neq 0, \beta$ are constants such that $a\beta+c=0$.

Conversely, if $f$ and $g$ have the forms in Theorem \ref{t1.1} $(ii)$, it is easy to check that $f$ and $g$ satisfy the given equation. This proves Theorem \ref{t1.1} $(ii)$. The proof is thus complete.
\end{proof}

\section{{Proof of Theorem \ref{t1.2}}}
\begin{proof} The proof is similar to the proof of Theorem \ref{t1.1}, using the fact that $T(r,p)= O(log r)=o(T(r,f))$ for any polynomial $p$. Consequently (\ref{EqT1.3}) still holds. If $g$ is a transcendental entire function in $\mathbb{C}^n$, then from the proof of Case 1 of Theorem \ref{t1.1}, we get a contradiction. Next we suppose $g$ is a polynomial in $\mathbb{C}^n$. Let $\deg(g)=m$. Note that $f$ is a transcendental meromorphic function in $\mathbb{C}$. If $m\geq 2$, then from the proof of Case 2 of Theorem \ref{t1.1}, we easily get a contradiction. Hence $m=1$, i.e., $g$ is linear in $\mathbb{C}^n$.
\end{proof}

\vspace{0.1in}
{\bf Compliance of Ethical Standards:}\par

{\bf Conflict of Interest.} The authors declare that there is no conflict of interest regarding the publication of this paper.\par

{\bf Data availability statement.} Data sharing not applicable to this article as no data sets were generated or analysed during the current study.

\end{document}